\documentclass[12pt]{amsart}

\usepackage{listings}
\usepackage{geometry}
\usepackage{enumerate}

\newtheorem{theorem}{Theorem}[section]

\newtheorem{conjecture}[theorem]{Conjecture}
\newtheorem{proposition}[theorem]{Proposition}
\newtheorem{lemma}[theorem]{Lemma}
\newtheorem{remark}[theorem]{Remark}
\newtheorem{corollary}[theorem]{Corollary}

\newcommand{\Ric}{\operatorname{Ric}}
\newcommand{\wh}{\widehat}
\newcommand{\Lap}{\Delta}
\newcommand{\pr}{\partial}
\newcommand{\ol}{\overline}
\newcommand{\ddt}{\left.\frac{d}{dt}\right|_{t=0}}
\newcommand{\nn}{\nonumber}

\DeclareMathOperator{\vol}{vol}
\DeclareMathOperator{\area}{area}
\DeclareMathOperator{\tr}{tr}
\DeclareMathOperator{\diag}{diag}
\DeclareMathOperator{\id}{id}
\DeclareMathOperator{\grad}{grad}

\numberwithin{equation}{section}

\begin{document}

\title[Minimal hypersurfaces with small first eigenvalue]{Minimal hypersurfaces with small first eigenvalue in manifolds of positive Ricci curvature}

\author{Jonathan J. Zhu}

\address{Department of Mathematics, Harvard University, 1 Oxford Street\\
Cambridge, Massachusetts 02138,
United States of America}
\email{jjzhu@math.harvard.edu}

\maketitle

\begin{abstract}
In this paper we exhibit deformations of the hemisphere $S^{n+1}_+$, $n\geq 2$, for which the ambient Ricci curvature lower bound $\Ric \geq n $ and the minimality of the boundary are preserved, but the first Laplace eigenvalue of the boundary decreases. The existence of these metrics suggests that any resolution of Yau's conjecture on the first eigenvalue of minimal hypersurfaces in spheres would likely need to consider more geometric data than a Ricci curvature lower bound. 
\end{abstract}

\section{Introduction}

A minimal hypersurface $\Sigma^n$ in a Riemannian manifold $M^{n+1}$ is one for which the mean curvature function $H$ vanishes identically. The study of minimal hypersurfaces is one of the central topics in differential geometry, particularly in the case of ambient spaces having constant curvature. It is of great interest to determine how geometric data such as the spectrum of a minimal hypersurface depends on the ambient geometry (see \cite{schoen1994lectures} for some discussion). 

One outstanding open problem is the following conjecture of S.T. Yau:

\begin{conjecture}[Yau]
\label{conj:yau}
Let $S^{n+1}$ be the unit $(n+1)$-sphere with its standard round metric. Then the first eigenvalue of the Laplacian on a closed embedded minimal hypersurface $\Sigma^n\subset S^{n+1}$ is precisely $n$. 
\end{conjecture}

This important conjecture appears as problem 100 in Yau's 1982 problem list \cite{yau1982problem}, in his lectures \cite{yau1986nonlinear} and again in his more recent review \cite{yau2000review}. It is a natural conjecture since when the ambient space $M$ is the round unit sphere $S^{n+1}$, the minimal surface equation shows that the embedding $X:\Sigma^n \hookrightarrow S^{n+1}\hookrightarrow \mathbb{R}^{n+2}$ satisfies $\wh{\Lap} X+nX=0,$ where $\wh{\Lap}$ is the Laplacian on $\Sigma$. Consequently, the first eigenvalue must satisfy $\lambda_1(\Sigma) \leq n$, and it is not difficult to see that equality holds for $\Sigma$ a great sphere, for instance. In fact, to the author's knowledge, Yau's conjecture has also been verified for all known examples of minimal hypersurfaces in spheres - see for instance \cite{choe2006minimal} or more recently \cite{tang2013isoparametric}. 

In this paper we investigate the viability of approaching a proof of Yau's conjecture based on the positivity of Ricci curvature. The most significant progress towards a proof of the conjecture, due to Choi and Wang \cite{choi1983}, is indeed in this direction:
\begin{theorem}[Choi-Wang]
\label{thm:cw}
Let $\Sigma^n$ be a closed orientable embedded minimal hypersurface in a compact orientable manifold $(M^{n+1},g)$. Suppose that the ambient Ricci curvature satisfies $\Ric_g \geq kg$, $k>0$. Then $\lambda_1(\Sigma) \geq k/2$. 
\end{theorem}
\begin{remark}
\label{rmk:cw}
The conclusion $\lambda_1 \geq k/2$ in fact holds for compact manifolds $M$ with boundary $\pr M =\Sigma$, so long as the second fundamental form $A$ and first eigenfunction $\phi_1$ on the boundary satisfy $\int_\Sigma A(\wh{\nabla}\phi_1, \wh{\nabla}\phi_1) \geq 0.$ Indeed, the proof of Theorem \ref{thm:cw} relies on choosing the side of $\Sigma$ for which the latter condition is satisfied.

Also, Choi and Schoen \cite{CS} were later able to remove the orientability assumption.
\end{remark}
In particular, since the standard metric $\ol{g}$ on $S^{n+1}$ satisfies $\Ric_{\ol{g}} = n\ol{g}$, for minimal hypersurfaces in $S^{n+1}$ we have the lower bound $\lambda_1(\Sigma) \geq n/2$. Recently, Ding and Xin \cite{ding2013volume} were able to extend the argument to mean curvature flow self-shrinkers, assuming a lower bound for the ambient Bakry-\'Emery-Ricci tensor. 

The main problem we deal with in this paper is whether the Choi-Wang estimate may be improved by a factor of 2, thereby implying Yau's conjecture. Thus we ask if the following common generalisation of Conjecture \ref{conj:yau} and Theorem \ref{thm:cw} might hold:

\begin{conjecture}
\label{question}
Let $\Sigma^n$ be a closed embedded minimal hypersurface in a compact Riemannian manifold $(M^{n+1},g)$, for which $\Ric_g \geq kg$, $k>0$. Then $\lambda_1(\Sigma) \geq k$. 
\end{conjecture}

For $n=1$, Conjecture \ref{question} (and hence Yau's conjecture) holds by a classical theorem of Toponogov \cite{toponogov1959evaluation}. Henceforth we consider the case $n\geq2$.

In this paper we are able to find a real analytic deformation of the standard metric on the hemisphere $S^{n+1}_+$, $n\geq 2$, which decreases the first Laplace eigenvalue of the boundary $\pr S^{n+1}_+ \simeq S^n$ whilst preserving its minimality as well as the ambient Ricci curvature bound. This provides a counterexample to Conjecture \ref{question} in the class of manifolds with boundary, and is the main theorem of this paper:

\begin{theorem} 
\label{thm:main}
Let $n\geq 2$. Then there is a smooth metric $g$ on the hemisphere $S^{n+1}_+$ such that: 

\begin{itemize}
\item The ambient Ricci curvature satisfies $\Ric_g \geq ng$.
\item The boundary $\Sigma=\pr S^{n+1}_+$ is minimal with respect to $g$.
\item The first Laplace eigenvalue of the induced metric $\wh{g}$ on $\Sigma$ satisfies $\lambda_1(\Sigma)<n$.
\item The eigenfunction $\phi_1$ corresponding to $\lambda_1(\Sigma)$ satisfies $\int_\Sigma A(\wh{\nabla}\phi_1, \wh{\nabla}\phi_1) > 0,$ where $\wh{\nabla}$ is the gradient on $\Sigma$. 
\end{itemize}
\end{theorem}

The final point about the sign of the second fundamental form $A$ is significant in light of Remark \ref{rmk:cw}. In particular, our result shows that the Choi-Wang argument cannot be improved by a factor of 2 if one considers only manifolds with boundary, that is, only one side of $\Sigma$. Consequently, although we are not able to offer a direct counterexample to Conjecture \ref{question} in the class of closed ambient manifolds $M$, we believe that Theorem \ref{thm:main} provides good evidence that either a stronger property than positivity of Ricci curvature, or an argument more directly involving both sides of $\Sigma$, would be needed to improve Theorem \ref{thm:cw} by the desired factor of 2.\\

It is reasonable to then consider whether there exist deformations of the standard metric that, to first order, decrease the first boundary eigenvalue yet preserve the Einstein condition $\Ric_g = ng$, not just the lower bound on Ricci curvature. To do so we decompose the first variation of the metric into the conformal direction $fg + \mathcal{L}_\omega g$ and the transverse traceless direction $h$. Here $\mathcal{L}_\omega$ is the Lie derivative by $\omega$; for precise definitions one may consult Section \ref{sec:prelim}. We obtain the following partial result that under certain technical conditions, such deformations do not exist:

\begin{proposition}
\label{prop:einstein}
Let $M=S^{n+1}_+$, with equator $\Sigma=\pr S^{n+1}_+$. Consider an analytic variation $g=g(t)$ of the standard metric $g(0)=\ol{g}$ on $M$. Let $\nu$ be the outward unit normal on $\Sigma$, and decompose $\ddt g=f\ol{g} + \mathcal{L}_\omega \ol{g} + h$, where $h$ is transverse traceless, with $h(\nu,\cdot) =0$ on $\Sigma$. Assume that the normal component $v= (\iota_\nu \omega)|_\Sigma$ satisfies $\wh{\Lap}v+nv=0$, where $\wh{\Lap}$ is the Laplacian on $\Sigma$. 

Further suppose that the Ricci curvature of $M$ and the mean curvature of $\Sigma$ are fixed to first order, so that $\ddt (\Ric_g-ng)  =0 $ and $\ddt H(\Sigma, g)=0$. Then the induced metric $\wh{g}(t)$ on $\Sigma$ satisfies $\ddt \lambda_1(\Sigma, \wh{g})= 0$.
\end{proposition}

Of course for deformations preserving the Ricci curvature on the whole sphere $S^{n+1}$ the first eigenvalue must be fixed, since it is known that the round sphere is isolated in the space of Einstein metrics. 
Our results may be compared and contrasted to an earlier rigidity theorem, found by Hang and Wang \cite{hang2009rigidity} in their study of the Min-Oo conjecture:

\begin{theorem}[Hang-Wang]
Let $(M^{n+1},g)$ be a compact manifold with boundary $\pr M$. Suppose that $\Ric_g\geq ng$. Further suppose that $\pr M$ is isometric to $S^n$ with its standard metric and finally that $\pr M$ is convex in the sense that the second fundamental form satisfies $A\geq 0$. Then $(M,g)$ is isometric to the hemisphere $S^{n+1}$ with its standard metric.
\end{theorem}

The Min-Oo conjecture itself was shown to be false by Brendle, Marques and Neves \cite{brendle2011deformations}, who also used a deformation method (together with a gluing argument) to produce a metric on $S^{n+1}_+$ having scalar curvature $R>n(n+1)$ and totally geodesic boundary isometric to the standard $S^n$. 

Along these lines, it is perhaps interesting to note that a scalar curvature lower bound on $M$ cannot provide any control on the first Laplace eigenvalue of the hypersurface $\Sigma$, even if $\Sigma$ is totally geodesic:

\begin{remark} The standard metric on $M=S^1\times S^2(r)$ has constant positive scalar curvature $R=\frac{2}{r^2}$, but the totally geodesic surface $\Sigma = S^1\times S^1(r)$ has area $4\pi^2 r$. By the celebrated estimate of Yang-Yau \cite{yang1980eigenvalues} for Riemann surfaces $\Sigma$ of genus $g$, \begin{equation}\lambda_1(\Sigma)\area(\Sigma) \leq 8\pi(1+g),\end{equation} the surface therefore has first eigenvalue $\lambda_1(\Sigma) \leq \frac{4}{\pi r}$, which is much smaller than $R=\frac{2}{r^2}$ for small $r$. 
\end{remark}

Let us now outline the structure and main ideas of this paper. The spectrum of the Laplace operator is, in general, rather difficult to compute. In searching for counterexamples to Conjecture \ref{question}, we are thus led to consider deformations of the standard metric on the sphere or hemisphere, which has totally geodesic equator $\Sigma$. We will use perturbation theory for real analytic deformations to ensure that the spectrum varies smoothly.

Variation formulae for the relevant quantities $ \Ric_g$, $A$ and $\wh{\Lap}_\Sigma$, together with preliminaries on notation, conventions and Riemannian geometry, are collected in Section \ref{sec:prelim}. In Section \ref{sec:sphvar} we specialise the variation formulae to perturbations of the round metric, considering the conformal and transverse directions separately. The proof of Proposition \ref{prop:einstein} follows readily and is found in Section \ref{sec:einstein}. 

The proof of our main theorem, Theorem \ref{thm:main}, requires only the conformal direction. It is found in Section \ref{sec:main}, and proceeds in two main steps:

 The first step is to construct an explicit conformal factor $f$ such that under a conformal variation $\ddt g = fg$, the first eigenvalue on the boundary decreases whilst the Ricci curvature lower bound is preserved, to first order in $t$. Fixing an eigenfunction $\phi_1$ of $\wh{\Lap}|_\Sigma$ with corresponding eigenvalue $\lambda_1 = \lambda_1(\Sigma)$, the first variation in the conformal direction is formally given by $\ddt \lambda_1 = c_1\int_\Sigma f - c_2\int_\Sigma f \phi_1^2$, for constants $c_1\geq0 , c_2>0$. When $n\geq 3$ we have $c_1>0$, so we choose $f|_\Sigma=\psi$ to be nonpositive, very negative near the zero set of $\phi_1$, and small otherwise, so that the first term dominates and $\ddt \lambda_1<0$. We then extend $f$ to the remainder of the hemisphere $S^{n+1}_+$ by setting $f= F(r) \psi(\theta)$, where $r$ is the geodesic distance to the pole. For a suitable convex choice of $F$, the Ricci curvature lower bound will be preserved. When $n=2$ we have $c_1=0$, which necessitates a much more careful choice of $F(r)$ as well as bumping up $f$ by a positive constant. In this case we require a mathematically rigorous analysis involving numerical calculation at a large set of points to verify one of the inequalities coming from the variation of Ricci curvature. The details of this analysis are presented in the appendix. 

The second step in the proof of Theorem \ref{thm:main} is to correct the mean curvature of $\Sigma$, which varies under conformal change by $\frac{n}{2}\pr_r f|_\Sigma$ to first order, to zero. Since the first order variation in mean curvature under an ambient diffeomorphism is governed by the Jacobi operator, as long as our choice of $\pr_r f|_\Sigma$ is orthogonal to the (finite) kernel of this operator, we may fix the mean curvature to first order by perturbing in the direction of an ambient diffeomorphism. Using the exact formula for the mean curvature under conformal changes, we then introduce a lower order correction to $f$ that fixes the mean curvature exactly at zero.

Finally, keeping track of the first variation of the second fundamental form and explicitly solving the Jacobi equation with input $\pr_r f|_\Sigma$ allows us to show that the variation of $\int_\Sigma A(\wh{\nabla}\phi_1,\wh{\nabla}\phi_1)$ is positive. 

\subsection*{Acknowledgements}

The author would like to thank Prof. Bill Minicozzi for his encouragement, valued guidance and helpful discussions. The author is also grateful to Prof. Shing-Tung Yau for his generous support and interest. We thank Peter Smillie and Nick Strehlke for several elucidating conversations, as well as Xin Zhou for his comments on the manuscript. Finally the author would like to thank the referee for their careful review of the paper. This research was partially supported by NSF grant DMS-1308244.

\section{Preliminaries}
\label{sec:prelim}

\subsection{Riemannian geometry}
\label{sec:riem}

In this paper we consider compact Riemannian manifolds $(M^{n+1},g)$ with a closed smoothly embedded hypersurface $\Sigma^n$, $n\geq 2$. 

We use $\wh{g}$ to denote the induced metric on $\Sigma$, and will use hats to distinguish metric quantities $(\wh{g},\wh{\nabla},\cdots)$ on the hypersurface $\Sigma$ from the corresponding quantities on the ambient manifold $M$. We denote by $dV_g$ the volume form attached to $g$.

We will use the summation convention for repeated indices, reserving Latin letters for the indices $i=0,\cdots, n$ on $M$ and Greek letters for the indices $\alpha = 1,\cdots, n$ on the hypersurface $\Sigma$. Typically, we choose coordinates so that, on $\Sigma$, the $e_0=\pr_r$ direction corresponds to the outward unit normal $\nu$ and, near $\Sigma$, the corresponding coordinate $r$ should parametrise unit speed geodesics that meet $\Sigma$ orthogonally. 

We define the second fundamental form $A=A(\Sigma,g)$ of $\Sigma$ with respect to $g$ by $A(X,Y) = g( \nabla_X \nu, Y)$, which in our chosen coordinates has components given by $A_{\alpha\beta} = \Gamma_{\alpha 0 \beta}= - \Gamma_{0\alpha\beta}$. The mean curvature is then \begin{equation}H=H(\Sigma, g)=\tr_{\wh{g}} A = g^{\alpha\beta}A_{\alpha\beta}.\end{equation} Note that with this convention the round sphere in Euclidean space has positive mean curvature. A minimal hypersurface is one for which $H\equiv 0$.

We choose the convention for the Riemann curvature following Chow-Lu-Ni \cite{chowluni}: \begin{equation}R(e_i, e_j) e_k = R^l_{ijk} e_l,\label{eq:riem1}\qquad\text{and}\qquad R_{ijkl} = g_{lm} R^m_{ijk}.\end{equation} The Ricci tensor $\Ric$ is then the covariant 2-tensor with components $R_{ij} = R^p_{pij}$. We denote by $g^{-1}\Ric_g$ the $(1,1)$-Ricci tensor, and we understand the Ricci curvature bound $\Ric_g \geq kg$ to mean that the endomorphism $g^{-1}\Ric_g-k$ is positive semidefinite. 

It will be useful to record that the second fundamental form with respect to a conformal metric $e^f g$ is given (see \cite{besse2007einstein} for example) by \begin{equation} A(\Sigma, e^f g) = e^{f/2}A(\Sigma, g) + \frac{1}{2} e^{f/2} (\pr_0 f) \wh{g},\end{equation} and therefore the mean curvature in this conformal metric is \begin{equation}\label{eq:Hconf} H(\Sigma, e^f g) = e^{-f}\tr_{ \wh{g}} A(\Sigma, e^f g) = e^{-f/2}\left( H(\Sigma, g) + \frac{n}{2} \pr_0 f\right) .\end{equation}

\subsection{Operators on symmetric tensors}

Let $\Omega^1(M)$ be the space of smooth 1-forms on $M$, and let $\mathcal{S}^2(M)$ be the space of smooth symmetric 2-tensors on $M$. 

For $\omega\in \Omega^1(M)$, its divergence $\delta_g\omega$ is the smooth function given by $\delta_g \omega = \nabla^i \omega_i$. For $h\in \mathcal{S}^2(M)$, the $g$-trace of $h$ is given by $\tr_g h = g^{ij}h_{ij}$. The divergence of $h$ is a 1-form $\delta_g h$, with components given by $(\delta_g h)_j = \nabla^i h_{ij}$. We sometimes omit the subscripts when the metric is clear from context.

We will use the sign convention for the Laplacian that makes it a negative operator. That is, for arbitrary symmetric tensors, the rough Laplacian is defined by $\Lap = g^{ij}\nabla_i\nabla_j$. It is convenient to define the Hodge Laplacian, acting on 1-forms by \begin{equation}\label{eq:hodgelap}\Lap_H \omega_ i =\Lap \omega_i - R_i^j\omega_j\end{equation} and the Lichnerowicz Laplacian, defined for symmetric 2-tensors by \begin{equation}\label{eq:lichlap}\Lap_L h_{ij} = \Lap h_{ij} + 2R_{kij}^l h^k_l - R_i^k h_{jk} - R_j^k h_{ik}.\end{equation}

We say that $\lambda$ is an eigenvalue of $\wh{\Lap}$ with eigenfunction $\phi\neq 0$ if $\wh{\Lap}\phi = -\lambda \phi$. On the closed manifold $\Sigma$, the eigenvalues of $\wh{\Lap}$ are well-known to be discrete and, in our sign convention, nonnegative, so we may order the distinct eigenvalues \[ 0=\lambda_0 < \lambda_1 < \lambda_2 \cdots\] In particular, $\lambda_0=0$ has multiplicity 1, and $\lambda_1$ is the least nonzero eigenvalue of $\wh{\Lap}$. We will denote by $V_\lambda$ the (finite dimensional) eigenspace corresponding to $\lambda$. 

\subsection{Adjoint decompositions}
\label{sec:adjoint}

When $M$ is compact without boundary, an integration by parts shows that, up to a sign, $\delta:\Omega^1(M)\rightarrow C^\infty(M)$ is the $L^2$-adjoint of the exterior derivative $d:C^\infty(M)\rightarrow \Omega^1(M)$. Similarly, the adjoint of the divergence $\delta:\mathcal{S}^2(M)\rightarrow \Omega^1(M)$ acting on symmetric 2-tensors is given by $\delta^* = -\frac{1}{2}\mathcal{L},$ where $\mathcal{L}$ is the Lie derivative of $g$ by (the metric dual of) $\omega$, \begin{equation}\label{eq:lie}(\mathcal{L}\omega)_{ij} =\nabla_i\omega_j + \nabla_j \omega_i.\end{equation}

Before continuing, let us record the commutators of $\delta$ and $\mathcal{L}$ with the Laplacian. The following is a theorem of Lichnerowicz (see \cite{lichnerowicz1961propagateurs}, or \cite{delay2007}). 

\begin{proposition}
\label{prop:comm}
Suppose that $g$ is a metric with parallel Ricci curvature $\Ric_g$. Then \begin{equation} \delta_g \circ \Lap_L = \Lap_H \circ \delta_g\qquad\text{and}\qquad\Lap_L\circ\mathcal{L}= \mathcal{L}\circ\Lap_H.\end{equation}
\end{proposition}

Now, as in Besse \cite{besse2007einstein}, taking the adjoint of the map $(\omega,f)\mapsto \mathcal{L} \omega + fg$ leads to the well-known orthogonal decomposition \begin{equation}\mathcal{S}^2(M) = (C^\infty(M) g +\mathcal{L}(\Omega^1(M))) \oplus \mathring{\mathcal{S}}^2(M),\label{eq:decompwob}\end{equation} where $\mathring{\mathcal{S}}^2(M)$ is the space of transverse traceless deformations, that is, $h \in \mathcal{S}^2(M)$ satisfying $\delta_g h =0$ and $\tr_g h =0$. 

When $M$ has nonempty boundary $\pr M=\Sigma$, a boundary term appears in the integration by parts: \begin{equation}\label{eq:lieadj}\int_M h^{ij} (\mathcal{L}\omega+fg)_{ij} = \int_M(f\tr_g h-2\omega^j(\delta_g h)_j) + \int_\Sigma \omega^j h_{0j}.\end{equation} Thus the orthogonal decomposition becomes (see for example \cite{ting1977problem}) \begin{equation}\mathcal{S}^2(M) = (C^\infty(M) g + \mathcal{L}(\Omega^1(M))) \oplus \mathring{\mathcal{S}}^2(M),\label{eq:decompwb}\end{equation} where now we denote by $\mathring{\mathcal{S}}^2(M)$ the space of transverse traceless deformations which satisfy the additional boundary conditions $h_{0j}=h_{j0}=0$ on $\Sigma$, for all $j$. 

\subsection{Perturbation theory}
\label{sec:firstvariations}

Consider a variation $g(t)$ of the metric on $M$, with $g(0)=g$ and $g' = h$, for some symmetric 2-tensor $h \in \mathcal{S}^2(M)$. Here, and henceforth, a primed quantity will denote its first variation, that is, $(\cdot)' = \ddt(\cdot)$. To avoid confusion with this notation for variations, derivatives of functions $F$ of a single variable will later be denoted using dots, $\dot{F}$. 

Throughout, we will assume our variations are at least $C^2$ in $t$.

\subsubsection{Ricci curvature} The first variation $\Ric'_g=\Ric'_g(h)$ of the Ricci tensor satisfies (see \cite{besse2007einstein}) \begin{equation}\label{eq:dric} -2\Ric'_g= \Lap_L h+ \nabla^2 (\tr_g h) - \mathcal{L}\delta_g h.\end{equation} Since we are interested in the Ricci curvature as compared to the metric, we are more interested in the variation $(g^{-1}\Ric_g)'$; if $g$ is initially Einstein with $\Ric_g = kg$, then this variation is given by \begin{equation} (g^{-1}\Ric_g)' = g^{-1}(\Ric'_g - kh).\end{equation}

\subsubsection{Second fundamental form} The first variation of the second fundamental form $A=A(\Sigma,g(t))$ is not difficult to compute (see for example \cite{lott2012}; note that our sign conventions are different): 
\begin{equation}\label{eq:dA} A'_{\alpha\beta} = \frac{1}{2} (-\nabla_\alpha h_{0\beta} - \nabla_\beta h_{0\alpha} + \nabla_0 h_{\alpha\beta} + h_{00}A_{\alpha\beta}).\end{equation} Then the mean curvature $H=H(\Sigma,g(t))$ satisfies \begin{equation}\label{eq:dH} H'=-h^{\alpha\beta}A_{\alpha\beta} + g^{\alpha\beta}  A'_{\alpha\beta} = -\wh{\nabla}^\alpha h_{0\alpha} + \frac{1}{2} (\pr_0 (\tr_{\wh{g}} \wh{h}) - H h_{00}) ,\end{equation} where $\wh{h}$ is the projection of $h$ to $\mathcal{S}^2(\Sigma)$. When $\Sigma$ is totally geodesic, it is sometimes convenient to compute in the ambient space instead; in this case we have \begin{equation} H'  =-(\delta_g h)_0 + \frac{1}{2} \pr_0 h_{00} + \frac{1}{2}\pr_0 \tr_g h .\label{eq:dH2} \end{equation}

\subsubsection{Laplace operator}
\label{sec:varlaplace}
 Berger \cite{berger1973} computes the first variation of the Laplacian and its spectrum on a closed manifold (see also \cite{bleecker1980splitting} and \cite{soufi}). Applying Berger's results to $\Sigma$, we have \begin{equation}\label{eq:dLap0}\wh{\Lap}' \phi = -(\wh{h},\wh{\nabla}^2\phi) - ( \delta_{\wh{g}}\wh{h}, d_\Sigma \phi) + \frac{1}{2}(d_\Sigma(\tr_{\wh{g}} \wh{h}), d_\Sigma \phi).\end{equation} Here $(\cdot,\cdot)$ denotes the natural scalar product induced by $\wh{g}$. If $\psi,\phi$ are in the same initial eigenspace $V_\lambda$ then, after some manipulations, integration by parts yields \begin{equation}\label{eq:dLap}\langle \psi, \wh{\Lap}'\phi\rangle = -\frac{1}{2} \langle \wh{h}, \phi\wh{\nabla}^2 \psi+ \psi \wh{\nabla}^2\phi\rangle + \frac{1}{2} \langle \delta_{\wh{g}}\delta_{\wh{g}}\wh{h}, \psi\phi\rangle - \frac{1}{4}\langle \wh{\Lap}(\tr_{\wh{g}}\wh{h}),\psi\phi\rangle .\end{equation} Here, and for the remainder of this article, $\langle \cdot,\cdot\rangle$ shall denote the $L^2$ inner product on $\Sigma$. We write $\wh{\Lap}'_{V_\lambda} = \pi_{V_\lambda} \circ \wh{\Lap}|_{V_\lambda}:V_\lambda\rightarrow V_\lambda$ for the projection of the restriction of $\wh{\Lap}'$ to $V_\lambda$. In particular, we see from (\ref{eq:dLap}) that $\wh{\Lap}'_{V_\lambda}$ is symmetric.

Even under smooth variations of the metric, the Laplace eigenvalues may not evolve smoothly. However, by the Rellich-Kato perturbation theory \cite{rellich,kato}, under \textit{real analytic} variations of the metric $\wh{g}(t)$, the eigenvalues and eigenfunctions of the closed manifold $\Sigma$ do also vary analytically in $t$. In particular, we use the following statement as in Lemma 3.15 of Berger \cite{berger1973}:

\begin{lemma}[Berger]
\label{lem:analeig}
Let $(\Sigma,\wh{g})$ be a compact Riemannian manifold. Consider a family of metrics $\wh{g}(t)$, analytic in $t$, with $\wh{g}(0)=\wh{g}$. If $\lambda$ is an eigenvalue of $\wh{\Lap}_{\wh{g}(0)}$ with multiplicity $m$, then there exist families of scalars $\Lambda_i(t)$ and smooth functions $\varphi_i(t)$, for $i=1,\cdots,m$, each depending analytically on $t$, such that:
\begin{itemize}
\item $\wh{\Lap}_{\wh{g}(t)} \varphi_i(t) = -\Lambda_i(t) \varphi_i(t)$ for each $i$,
\item $\Lambda_i(0)=\lambda$ for each $i$, and
\item the $\{\varphi_i(t)\}$ are $L^2(\Sigma,\wh{g}(t))$-orthonormal for all $t$. 
\end{itemize}
\end{lemma}

Given such families of orthonormal eigenfunctions $\varphi_i(t)$, differentiating the condition that the $\varphi_i$ are orthonormal we have that
$\langle \phi_i,\wh{\Lap}'\phi_j\rangle=0$ for $i\neq j$, and
\begin{equation}\label{eq:dL1}\lambda'_i = -\langle \phi_i,\wh{\Lap}'\phi_i\rangle=\langle \wh{h}, \phi_i \wh{\nabla}^2\phi_i\rangle +\frac{1}{2} \langle \delta_{\wh{g}} \wh{h} , d_\Sigma(\phi_i^2)\rangle- \frac{1}{4}\langle d_\Sigma(\tr_{\wh{g}} \wh{h}), d_\Sigma(\phi_i^2)\rangle.\end{equation}

\subsection{Geometry of round spheres}
\label{sec:round}

We realise $S^{n+1}$ as the unit sphere in $\mathbb{R}^{n+2}$, which has coordinate functions $x_0,\cdots, x_{n+1}$. We consider the upper hemisphere \begin{equation}S^{n+1}_+ = \{ x\in S^{n+1}:x_{n+1} \geq 0\},\end{equation} which has boundary given by the great sphere \begin{equation}\Sigma =\pr S^{n+1}_+= \{x\in S^{n+1}:x_{n+1}=0\}.\end{equation}

We will use $\ol{g}$ to denote the standard round metric as appropriate from context. For the standard metric $\ol{g}$, the curvature tensor is given by $R_{ijkl} = \ol{g}_{il}\ol{g}_{jk} - \ol{g}_{ik}\ol{g}_{jl}.$ Thus $\ol{g}=\ol{g}_{S^{n+1}}$ satisfies $\Ric_{\ol{g}} = n\ol{g}$, and so the variation of the Ricci tensor at $g(0)=\ol{g}$ satisfies \begin{equation}\label{eq:einsph}-2(\Ric'_g - nh) = \Lap_L h + \nabla^2 (\tr_{\ol{g}} h) - \mathcal{L}\delta_{\ol{g}} h + 2nh.\end{equation} Moreover, the Hodge Laplacian acts as $\Lap_H \omega = \Lap \omega -n\omega,$ and the Lichnerowicz Laplacian $\Lap_L$ acts simply as \begin{equation}\label{eq:lichlapsph}\Lap_Lh = \Lap h-2(n+1)h +2(\tr_{\ol{g}} h)\ol{g}.\end{equation}

Typically we choose coordinates so that the standard metric $\ol{g}$ can be written (again, with $e_0 = \pr_r$) as the warped product \begin{equation}\label{eq:warp1}\ol{g}_{S^{n+1}} = dr^2 + \sin^2{r} \,\ol{g}_{S^n}.\end{equation} In terms of the coordinate functions on $\mathbb{R}^{n+2}$, we take $x_{n+1}=\cos r$, so that the equator $\Sigma$ is the level set $r=\pi/2$. Note that this is consistent with Section \ref{sec:riem}.

In these coordinates, the Christoffel symbols of $\ol{g}_{S^{n+1}}$ are given by \begin{eqnarray}\label{eq:christoffel}&\Gamma^0_{00} = \Gamma^0_{\alpha0} = \Gamma^\gamma_{00} =0,\qquad\Gamma^0_{\alpha\beta} = -\sin r \cos r \,\wh{g}_{\alpha\beta},\nn \\& \Gamma^\gamma_{\alpha 0} = \delta_\alpha^\gamma\cot r,\qquad\Gamma^\gamma_{\alpha\beta} = \wh{\Gamma}^\gamma_{\alpha\beta},\end{eqnarray} where $\wh{\Gamma}$ are the Christoffel symbols for $\ol{g}_{S^n}$ and $\delta_\alpha^\gamma$ is the Kronecker delta. 

The warped product structure (\ref{eq:warp1}) also gives a decomposition of symmetric 2-tensors $h\in \mathcal{S}^2(M)$ as in Delay's paper \cite{delay2007}, \begin{equation}\label{eq:decomp2t} h = u \, dr^2 + \xi \otimes dr + dr \otimes \xi + \wh{h}.\end{equation} That is, if $S(r)\subset S^{n+1}$ is the geodesic $n$-sphere at a fixed $r$, then $u$ is the smooth function on $S(r)$ given by $u=h_{00}$, $\xi$ is the 1-form on $S(r)$ given by $\xi_\alpha = h_{\alpha0}$ and $\wh{h}$ is the projection of $h$ to $\mathcal{S}^2(S(r))$ so that $\wh{h}_{\alpha\beta}=h_{\alpha\beta}$. 

Later, we will also use the warped product form iteratively, that is, we choose a coordinate $s$ on $S^n$ so that the standard metric on $S^{n+1}$ becomes \begin{equation}\label{eq:warp2}\ol{g}_{S^{n+1}} = dr^2 + \sin^2{r}\,ds^2+\sin^2{r}\,\sin^2{s} \,\ol{g}_{S^{n-1}}.\end{equation} In terms of coordinates, recalling that earlier we took $x_{n+1}=\cos r$ where $e_0=\pr _r$, we now take $x_n = \sin r \cos s$ with $e_1 = \pr_s$. 

The eigenfunctions of the Laplacian $\wh{\Lap}$ on $S^n$ with the standard metric $\ol{g}_{S^n}$ are well-known; they are the restrictions of homogeneous harmonic polynomials on $\mathbb{R}^{n+1}$, which we have realised as the subspace $x_{n+1}=0$ of $\mathbb{R}^{n+2}$. The eigenvalues are then given by $d(d+n-1)$ for the homogeneous harmonic polynomials of degree $d$. 

In particular, the first eigenvalue of $\wh{\Lap}$ is $n$ with multiplicity $n+1$, and the corresponding eigenspace $V_n$ has a canonical orthonormal basis proportional to the remaining coordinate functions, which we denote \begin{equation}\phi_{1,i}=\frac{x_i}{\sqrt{C_{n}}},\qquad i=0,\cdots,n\end{equation} where $C_{n}  =  \frac{1}{n+1} \gamma_n$ and $\gamma_n$ is the volume of $S^n$ with standard metric $\ol{g}$. It will be useful to compute $C_n$ and related integrals directly:

\begin{lemma}
\label{lem:sphint}
For all nonnegative integers $k$, we have \begin{equation}\int_{S^n} (1-x_n^2)^k\, dV_{\ol{g}} = B(k+\frac{n}{2},\frac{1}{2})\gamma_{n-1},\end{equation} \begin{equation}\int_{S^n} x_n^2 (1-x_n^2)^k\, dV_{\ol{g}} = B(k+\frac{n}{2},\frac{3}{2})\gamma_{n-1},\end{equation} where $B(x,y)$ is the beta function and again $\gamma_{n-1} = \vol(S^{n-1},\ol{g})$.
\end{lemma}
\begin{proof}
In our coordinates, we have \begin{equation}\int_{S^n} (1-x_n^2)^k\, dV_{\ol{g}} = \gamma_{n-1} \int_0^\pi \sin^{2k} s \sin^{n-1} s \, ds.\end{equation} Making the substitution $y= \sin^2 s$, the right hand side is then given by \begin{equation}2\gamma_{n-1}\int_0^{\frac{\pi}{2}} (\sin s)^{2k+n-1} \, ds = \gamma_{n-1}\int_0^1 y^{k+\frac{n-1}{2}-\frac{1}{2}} (1-y)^{-\frac{1}{2}} dy,\end{equation} and we recognise the last term as the desired beta integral.

Similarly, we have \begin{equation}\int_{S^n} x_n^2 (1-x_n^2)^k \, dV_{\ol{g}}= \gamma_{n-1} \int_0^\pi \sin^{2k} s \cos^2 s \sin^{n-1}s \, ds,\end{equation} and the same substitution in the right hand side gives \begin{equation} 2 \gamma_{n-1}\int_0^{\frac{\pi}{2}} (\sin s)^{2k+n-1}\cos^2 s \,ds = \gamma_{n-1}\int_0^1 y^{k+ \frac{n-1}{2}-\frac{1}{2}} (1-y)^{1-\frac{1}{2}}dy,\end{equation} which we again recognise as the claimed beta integral.
\end{proof}

It is well-known that the first eigenfunctions $\phi\in V_n$ satisfy the Hessian equation $\wh{\nabla}^2 \phi = -\phi \ol{g}_{S^n}.$ Moreover, since $|x|=1$ on the sphere, we can decompose the squares of first eigenfunctions in terms of zeroth and second degree eigenfunctions - in particular, $x_i^2 = \frac{1}{n+1} \left(1+ \sum_{j\neq i} (x_i^2-x_j^2)\right).$ Since the second eigenvalue is $2(n+1)$ we then have \begin{equation}\label{eq:lapphi2}\wh{\Lap} \phi_{1,i}^2 = \frac{2}{C_{n}} - 2(n+1) \phi_{1,i}^2.\end{equation} Also, for $i\neq j$, the product $\phi_{1,i}\phi_{1,j}$ is itself a second degree eigenfunction: \begin{equation}\label{eq:lapphi3} \wh{\Lap}(\phi_{1,i}\phi_{1,j}) = -2(n+1)\phi_{1,i}\phi_{1,j}.\end{equation}

\section{Deformations of the round metric}
\label{sec:sphvar}

In this section we fix $M$ to be $S^{n+1}$ or $S^{n+1}_+$, and consider deformations $g=g(t)$ of the standard metric $g(0)=\ol{g}$. Throughout this section we set $h=g'$. 

\subsection{Conformal direction}
\label{sec:conformal}

We first consider the conformal direction $C^\infty(M)\ol{g} + \mathcal{L}(\Omega^1(M)),$ beginning with the diffeomorphism part:

\begin{lemma}
\label{lem:diffeo}
Suppose $h = \mathcal{L}\omega$, for $\omega\in\Omega^1(M)$. Then:
\begin{itemize}
\item The Ricci curvature is preserved to first order, that is, $\Ric'_g - n h =0$.
\item The variation $\wh{\Lap}'$ of the Laplacian on $\Sigma$ acts on each $V_\lambda$ by $\wh{\Lap}'_{V_\lambda}=0$.
\item The second fundamental form $A=A(\Sigma,g(t))$ varies as $ A'= -\wh{\nabla}^2 v - v\ol{g}_{S^n}.$ Consequently, $H' =-(\wh{\Lap}+n)v,$ where $v$ is the smooth function on $\Sigma$ given by $v = \omega_0 |_\Sigma = \iota_\nu \omega$. 
 \end{itemize}
\end{lemma}
\begin{proof}
For the first claim, take $h= \mathcal{L} \omega $, for $\omega \in \Omega^1(M)$. Then if $X=\omega^\sharp$ is the vector field dual to $\omega$, we have $h = \mathcal{L}_X \ol{g} = \ddt \Phi_t^* \ol{g},$ where $\Phi_t$ is the one-parameter group of diffeomorphisms of $M$ generated by $X$. Therefore \begin{equation}\Ric'_g = \ddt \Phi_t^*\Ric_{ g}  = \mathcal{L}_X \Ric_{\ol{g}} = n\mathcal{L}_X \ol{g} = nh.\end{equation}

Note that this argument is still valid on the hemisphere $S^{n+1}_+$, for example by taking an arbitrary extension of $X$ to the whole sphere $S^{n+1}$. One may also verify this claim by direct computation using Proposition \ref{prop:comm} and (\ref{eq:einsph}).

Now since the equator $\Sigma$ is totally geodesic, we have $\wh{h}_{\alpha\beta}  = \wh{\nabla}_\alpha \omega_\beta + \wh{\nabla}_\beta\omega_\alpha=(\wh{\mathcal{L}}\wh{\omega})_{\alpha\beta}$, where $\wh{\omega}$ is the projection of $\omega$ to $\Omega^1(\Sigma)$. Since the Laplace spectrum is a geometric quantity, formally this should mean that the spectrum is fixed to first order as above. More generally we deal with the quantity $\wh{\Lap}'$ by direct computation: By commuting indices we have $\delta_{\wh{g}}\delta_{\wh{g}}\wh{h} = 2(\wh{\Lap}+n-1)\delta_{\wh{g}}\wh{\omega}$. Also $\tr_{\wh{g}} \wh{h} = 2\delta_{\wh{g}} \wh{\omega}$. For $\psi,\phi \in V_\lambda$, then integrating by parts we can compute \begin{eqnarray}\nn -\frac{1}{2}\langle \wh{\mathcal{L}}\wh{\omega}, \psi\wh{\nabla}^2 \phi + \phi \wh{\nabla}^2\psi\rangle &=& \langle \wh{\omega}, \delta_{\wh{g}} (\psi\wh{\nabla}^2 \phi + \phi \wh{\nabla}^2\psi)\rangle \\&=& \langle \wh{\omega} , \frac{1}{2} d_\Sigma \wh{\Lap}(\psi\phi) + (n-1) d_\Sigma(\psi\phi)\rangle.\end{eqnarray} Substituting into (\ref{eq:dLap}) and integrating by parts again then gives that $\langle \psi,\wh{\Lap}'\phi\rangle=0$.

The variations of $A$ and $H$ can be essentially found in \cite{huisken1999geometric}, but since our setting is slightly different we include the computations here for completeness. Indeed, commuting indices we compute \begin{eqnarray}\nn \nabla_\alpha h_{0\beta} + \nabla_\beta h_{0\alpha} -\nabla_0 h_{\alpha\beta} &=& \nabla_\alpha \nabla_\beta \omega_0 + \nabla_\beta \nabla_\alpha \omega_0 - R^l_{\alpha0\beta}\omega_l - R^l_{\beta0\alpha}\omega_l \\&=& \nabla_\alpha \nabla_\beta \omega_0 + \nabla_\beta\nabla_\alpha \omega_0 +2 \omega_0 \ol{g}_{\alpha\beta} .\end{eqnarray} Since $\Sigma$ is totally geodesic, on $\Sigma$ we have $\nabla_\alpha\nabla_\beta \omega_0 = \wh{\nabla}_\alpha\wh{\nabla}_\beta v$, so equation (\ref{eq:dA}) gives $ A'_{\alpha\beta}= -\wh{\nabla}_\alpha\wh{\nabla}_\beta v - v\ol{g}_{\alpha\beta},$ and finally equation (\ref{eq:dH}) gives $ H' = \ol{g}^{\alpha\beta}  A'_{\alpha\beta} = -\wh{\Lap}v - nv.$ 
\end{proof}

Now we consider the conformal factors:

\begin{lemma}
\label{lem:conformal}
Suppose $h= f\ol{g}$, for $f\in C^\infty(M)$. Then:
\begin{itemize}
\item The Ricci curvature varies as $-2(\Ric'_g-nh) = (\Lap f+2nf)\ol{g}+(n-1)\nabla^2 f.$
\item The variation $\wh{\Lap}'$ of the Laplacian on $\Sigma$ acts on each eigenspace $V_\lambda$ by \begin{equation} \langle \psi,\wh{\Lap}'\phi\rangle = \lambda\langle f,\psi\phi\rangle -\frac{n-2}{4}\langle \wh{\Lap}f,\psi\phi\rangle ,\end{equation} where $\psi,\phi\in V_\lambda$.
\item The second fundamental form $A=A(\Sigma,g(t))$ varies as $ A'  = \frac{1}{2}(\pr_0 f )\ol{g}_{S^n}$ and hence $H' = \frac{n}{2}\pr_0 f$.. 
\end{itemize}
\end{lemma}
\begin{proof}
Take $h=f\ol{g}$, for $f\in C^\infty(M)$. Then $\mathcal{L}\delta_{\ol{g}} h  = \mathcal{L} df= 2 \nabla^2 f$ and $\Lap_L h = (\Lap f)g$, so we have $-2(\Ric'_g-nh) =(\Lap f + 2nf)\ol{g} + (n-1)\nabla^2 f.$

Now note that $\delta_{\ol{g}} h = df$. Plugging into equation (\ref{eq:dA}) immediately gives that $ A'_{\alpha\beta} = \frac{1}{2}(\pr_0 f) \ol{g}_{\alpha\beta}.$ Then by (\ref{eq:dH}) and since $\Sigma$ is totally geodesic, we have that $H' = \ol{g}^{\alpha\beta}A'_{\alpha\beta} = \frac{n}{2}\pr_0 f.$

For the variation of the Laplacian, we note that $\tr_{\wh{g}}\wh{h} = nf$, $\delta_{\wh{g}}\delta_{\wh{g}}\wh{h} = \wh{\Lap}f$ and \begin{equation}\langle \wh{h}, \psi \wh{\nabla}^2\phi\rangle = \langle f,\tr_{\wh{g}}(\psi \wh{\nabla}^2 \phi) \rangle= \langle f, \psi \wh{\Lap}\phi\rangle = -\lambda\langle f,\psi\phi\rangle.\end{equation} Similarly $\langle \wh{h}, \phi\wh{\nabla}^2\psi\rangle = -\lambda \langle f,\psi\phi\rangle$. Substituting into (\ref{eq:dLap}) gives the result. 

\end{proof}

Using (\ref{eq:lapphi2}) and (\ref{eq:lapphi3}) for the Laplacian of products of the first eigenfunctions $\phi_{1,i}$ on the standard sphere $S^n$ and integrating the second term by parts, we may simplify:
\begin{corollary}
\label{cor:dL1}
Let $h= f\ol{g}$ as above. Recall that the first eigenfunctions on $(\Sigma,\ol{g}_{S^n})$ are given by $\phi_{1,i} = \frac{x_i}{\sqrt{C_n}}$. We have \begin{equation}  \langle \phi_{1,i},\wh{\Lap}'\phi_{1,j}\rangle = -\frac{n-2}{2} \frac{\langle f,1\rangle}{C_{n}}\delta_{ij} + \frac{1}{2}(n+2)(n-1)\langle f,\phi_{1,i}\phi_{i,j}\rangle.\end{equation}
\end{corollary}

\subsection{Transverse direction}
\label{sec:transverse}

Now we consider the remaining variations, those with $g'=h \in \mathring{\mathcal{S}}^2(M)$. Recall that $\mathring{\mathcal{S}}^2(M)$ is the space of transverse traceless symmetric 2-tensors $h$ satisfying $\tr_{\ol{g}} h= 0$, $\delta_{\ol{g}} h=0$, as well as the extra boundary conditions $h_{0j}=0$ on $\Sigma=\pr M$ in the case of the hemisphere $M=S^{n+1}_+$. 

We will use some formulae from \cite{delay2007}, in which the action of the Lichnerowicz Laplacian on transverse traceless tensors for more general warped product spaces is computed with respect to the decomposition (\ref{eq:decomp2t}). The relevant formulae are, for $h\in\mathcal{S}^2(M)$ satisfying $\delta_{\ol{g}} h=0$, $\tr_{\ol{g}} h=0$, \begin{equation}\label{eq:lapl00}\Lap_L h_{00}= \pr_0^2 u+(n+4)\cot{r}\,\pr_0 u+2(n+1)\left(\cot^2 r-1\right)u + \frac{1}{\sin^2 r} \wh{\Lap}u, \end{equation}  \begin{eqnarray}\label{eq:lapl0a}\Lap_L h_{\alpha0} &=& \pr_0^2\xi_\alpha + n\cot r\,\pr_0\xi_\alpha + \left((n+2)\cot^2 r+2-n\right)\xi_\alpha \nn\\&&+ \frac{1}{\sin^2 r}(\wh{\Lap} \xi_\alpha-\Ric_{\wh{g}} \xi_\alpha)+ 2\cot r\, \pr_\alpha u.\end{eqnarray} Note that in these formulae $\wh{\Lap}$ and $\wh{\nabla}$ correspond to the standard metric $\ol{g}_{S^n}$ rather than the induced metric on $S(r)$, which differs in scale. In any case, we will only apply them at the boundary $\Sigma = S(\frac{\pi}{2})$, at which these metrics coincide.

First we recall that the Lichnerowicz Laplacian preserves the transverse traceless condition (see \cite{besse2007einstein} for example); for the round metric this is easy to verify manually:
\begin{lemma}
\label{lem:split}
Let $M=S^{n+1}$ or $M=S^{n+1}_+$. Suppose that $h\in {\mathcal{S}}^2(M)$ satisfies $\tr_{\ol{g}}h=0$ and $\delta_{\ol{g}}h=0$. Then $\tr_{\ol{g}}\Lap_L h =0$ and $\delta_{\ol{g}}\Lap_L h=0$. 
\end{lemma}
\begin{proof}
 Since $\tr_{\ol{g}}h=0$, by equation (\ref{eq:lichlapsph}) we have $\Lap_L h = \Lap h -2(n+1)h$ and hence $\tr_{\ol{g}} \Lap_L h = \tr_{\ol{g}}\Lap h = \Lap(\tr_{\ol{g}}h)=0.$ Proposition \ref{prop:comm} gives $\delta_{\ol{g}} \Lap_L h = \Lap_H \delta_{\ol{g}} h =0.$
\end{proof}

Now we give variation formulae for deformations by transverse traceless symmetric 2-tensors. In particular we do not yet assume any boundary conditions for these formulae. 

\begin{lemma}
\label{lem:transverse}
Suppose that $h\in \mathcal{S}^2(M)$ satisfies $\tr_{\ol{g}}h=0$ and $\delta_{\ol{g}}h=0$. Then:
\begin{itemize}
\item The Ricci curvature varies as $-2(\Ric'_g-nh) = \Lap_L h + 2nh.$ 
\item The variation $\wh{\Lap}'$ of the Laplacian on $\Sigma$ acts on the first eigenspace $V_n$ by \begin{equation}\langle \psi, \wh{\Lap}'\phi\rangle = \langle \psi\phi, \frac{1}{2}\pr_0^2 u -\frac{n+3}{2}u +\frac{1}{4}\wh{\Lap}u\rangle\end{equation} where $\psi,\phi\in V_n$ and $u$ is as in the decomposition (\ref{eq:decomp2t}). 
\item Finally, the mean curvature $H=H(\Sigma,g(t))$ varies as $H' =\frac{1}{2}\pr_0 u.$ 
\end{itemize}
\end{lemma}
\begin{proof}
The formulae for $\Ric'_g$ and $H'$ follow immediately from equations (\ref{eq:einsph}) and (\ref{eq:dH2}) using the transverse traceless property.

For the variation of $\lambda$, note that since $\Sigma$ is totally geodesic, in our coordinates we have $\tr_{\wh{g}} \wh{h}= - h_{00} = -u$ and $(\delta_{\wh{g}} \wh{h})_\beta = \nabla^\alpha h_{\alpha\beta} =  - \nabla^0 h_{0\beta},$ on $\Sigma$. Commuting indices we have that, again on $\Sigma$, \begin{eqnarray}\nn \delta_{\wh{g}}\delta_{\wh{g}} \wh{h}  &=&-\ol{g}^{\alpha\beta} \nabla_\alpha \nabla_0 h_{0\beta}=- \ol{g}^{\alpha\beta}(\nabla_0 \nabla_\alpha h_{0\beta} - R^p_{\alpha00}h_{p\beta} - R^p_{\alpha0\beta}h_{0\beta})\\\nn&=& \nabla_0 \nabla_0 h_{00} + \ol{g}^{\alpha\beta}h_{\alpha\beta} - \ol{g}^{\alpha\beta}\ol{g}_{\alpha\beta}u\\&=&  \pr_0^2u -(n+1)u. \end{eqnarray}

Now since $\wh{\nabla}^2 \phi = -\phi\wh{g}$, we have $\langle \wh{h}, \psi \wh{\nabla}^2 \phi\rangle = -\langle \tr_{\wh{g}} \wh{h}, \psi\phi\rangle = \langle u, \psi\phi\rangle.$ Similarly $\langle \wh{h}, \phi\wh{\nabla}^2\psi\rangle = \langle u,\psi\phi\rangle$. Substituting into (\ref{eq:dLap}) gives the result.

\end{proof}

\begin{corollary}
\label{lem:transein}
Let $M=S^{n+1}_+$, with equator $\Sigma$. Suppose that $g'=h\in \mathcal{S}^2(M)$ satisfies $\tr_{\ol{g}}h=0$, $\delta_{\ol{g}}h=0$ and $h_{00}|_\Sigma=0$. If we additionally have either $\Ric'_g \geq nh$ or $\Ric'_g \leq nh$, then $\wh{\Lap}'_{V_n} =0$. 
\end{corollary}
\begin{proof}
 By Lemma \ref{lem:transverse}, the Ricci condition gives $\Lap_L h + 2nh \leq 0$ or $\Lap_L h + 2nh \geq 0$ respectively. But $\Lap_L h + 2nh$ is traceless if $h$ is transverse traceless by Lemma \ref{lem:split}. Therefore if $\Lap_L h+2nh$ has a sign, then it must in fact be zero. 

We now use the explicit computation of the Lichnerowicz Laplacian. In particular, (\ref{eq:lapl00}) reduces on the boundary $r=\pi/2$ to give \begin{equation}0=(\Lap_L h+2nh)_{00} = \pr_0^2 u -2 u + \wh{\Lap}{u}.\end{equation} Since $u=h_{00}$ vanishes on $\Sigma$, this implies that we also have $\pr_0^2 u|_\Sigma=0$. Thus Lemma \ref{lem:transverse} implies that $\wh{\Lap}'_{V_n}=0$ as claimed. 
\end{proof}

Corollary \ref{lem:transein} implies in particular that, to first order, variations $h\in\mathring{\mathcal{S}}^2(S^{n+1}_+)$ cannot affect the first boundary eigenvalues whilst increasing Ricci curvature. On the whole sphere the variations $h\in \mathring{\mathcal{S}}^2(S^{n+1})$ the argument is simpler still:

 \begin{proposition}
 \label{prop:ttsign}
 Let $M$ be the round sphere $S^{n+1}$ with variation $g'=h\in\mathring{\mathcal{S}}^2(M)$. If $\Ric'_g \geq nh$, or if $\Ric'_g \leq nh$, then $h=0$. 
 \end{proposition}
 \begin{proof}
Arguing as in Corollary \ref{lem:transein}, the Ricci conditions both imply $\Lap_L h +2nh=0$. But for the round metric on $S^{n+1}$, there are no nontrivial $h \in \mathring{\mathcal{S}}^2(M)$ for which $\Lap_L h+2nh=0$. In fact, from Boucetta's computation of the spectrum of the Lichnerowicz Laplacian on spheres \cite{boucetta1999,boucetta2009}, it is known that the least eigenvalue of $\Lap_L$ acting on the transverse traceless space $\mathring{\mathcal{S}}^2(S^{n+1})$ is $4(n+1)$. 
\end{proof}

\subsection{Einstein deformations}
\label{sec:einstein}

In this section we prove Proposition \ref{prop:einstein} by using our analysis of the conformal and transverse directions given by the orthogonal decompositions (\ref{eq:decompwob}) and (\ref{eq:decompwb}). 

\begin{proof}[Proof of Proposition \ref{prop:einstein}]
Recall that $M=S^{n+1}_+$ with equator $\Sigma= \pr M$. Write $g' =f\ol{g} + \mathcal{L}\omega + h$, where $h \in \mathring{\mathcal{S}}^2(M)$. Then by Lemmas \ref{lem:diffeo}, \ref{lem:conformal} and \ref{lem:transverse}, the Ricci condition gives \begin{equation}\label{eq:pfeinstein1}(\Lap f + 2nf)\ol{g} + (n-1)\nabla^2 f + \Lap_L h+2nh=0.\end{equation} Taking the trace gives $\Lap f = -(n+1)f$ on $M$, and hence \begin{equation} (n-1)(\nabla^2 f + f\ol{g}) + \Lap_L h + 2nh=0.\end{equation}

Now assume that $H'=0$ and that $v=\omega_0|_\Sigma$ satisfies $\wh{\Lap}v+nv=0$. Using the decomposition (\ref{eq:decomp2t}), we note that any $h\in\mathring{\mathcal{S}}^2(M)$ has, on $\Sigma$, that \begin{equation}\pr_0 u = (\delta_{\ol{g}} h)_0 - \wh{\nabla}^\alpha h_{0\alpha}= 0.\end{equation} The mean curvature condition then gives that $\pr_0 f|_\Sigma=0$. But then we may reflect $f$ to a $C^1$ (hence $C^\infty$) solution of $\Lap f+(n+1)f=0$ on the closed manifold $S^{n+1}$. Such eigenfunctions are known to be restrictions of coordinate functions on $\mathbb{R}^{n+2}$, which satisfy $\nabla^2 f = -f\ol{g}$ on $S^{n+1}$. 

Therefore $g' = -\nabla^2 f + \mathcal{L}\omega + h$, where $h\in \mathring{\mathcal{S}}^2(M)$ satisfies $\Lap_L h +2nh =0$. So noting that $\nabla^2 f = \frac{1}{2}\mathcal{L}df$, Lemma \ref{lem:diffeo} and Corollary \ref{lem:transein} give that $\wh{\Lap}'_{V_n}=0$. We conclude that $\lambda_1(\Sigma)$ is fixed to first order by Lemma \ref{lem:analeig}. 
\end{proof}

\section{Deformations of the hemisphere with Ricci curvature bound}
\label{sec:main}

In this final section, we consider the case of the hemisphere $M=S^{n+1}_+$ with boundary $\Sigma=\pr S^{n+1}_+$, with the goal of proving Theorem \ref{thm:main}.

\subsection{An explicit ambient conformal factor}
\label{sec:explicit}

Our first step is to construct an explicit smooth function $f$ on $M=S^{n+1}_+$ so that the variation $h=f\ol{g}$ decreases the first Laplace eigenvalue on the boundary, whilst preserving the lower bound on Ricci curvature to first order. Specifically, this section contains the proof of the following:

\begin{proposition}
\label{prop:explicit}
Let $M=S^{n+1}_+$, $\Sigma =\pr S^{n+1}_+$, for $n\geq 2$. There exists a smooth function $f\in C^\infty(M)$ such that, if $g=g(t)$ is a variation of the standard metric $g(0)=\ol{g}$, with $g' = f\ol{g}$, then:
\begin{itemize}
\item The variation of Ricci curvature satisfies $(g^{-1}\Ric_{g})' \geq 0$ on $M$.
\item The variation $\wh{\Lap}'$ of the Laplacian acting on the first eigenspace $V_n$ of $(\Sigma, \ol{g}_{S^n})$ is diagonal with respect to the basis $\{\phi_{1,i}\}$, \begin{equation} \wh{\Lap}'_{V_n} = \diag(-\mu_{1},\cdots,-\mu_{n}),\end{equation} with $\mu_i>0$ for $i<n$ and $\mu_n <0$. 
\end{itemize}
Moreover, $f$ may be chosen so that $\pr_0 f|_\Sigma$ is orthogonal to the space $V_n$. 

\end{proposition}

Each $\mu_i$ should be formally regarded as the first variation of the eigenvalue attached to $\phi_{1,i}$. They will genuinely be the first variations under the analytic variation that we present later, in the proof of Theorem \ref{thm:main}.

To construct the function $f$, we work with the coordinates $r,s$ discussed in Section \ref{sec:round}, so that the standard metric takes the form (\ref{eq:warp2}): \begin{equation}\ol{g}_{S^{n+1}} = dr^2 + \sin^2{r}\,ds^2+\sin^2{r}\,\sin^2{s} \,\ol{g}_{S^{n-1}},\end{equation} where $0\leq r \leq \pi/2$, $0\leq s\leq \pi$. Recall that in our realisation of the hemisphere $M=S^{n+1}_+$, we had $\cos r = x_{n+1}|_M$, $\sin r \cos s = x_n|_M$. 

We will use functions of the form \begin{equation}\label{eq:fform}f = a + F(r) \psi(s),\end{equation} where $a$ is some constant, $F$ is a polynomial in $r^2$ with \begin{equation}F(0)=0,\qquad F(\pi/2)=b>0, \qquad \dot{F}(\pi/2)=c>0\end{equation} and $\psi(s) = -\sin^{2k} s$ for some positive integer $k$. The restriction to even powers of $r$ and $\sin s$ ensures that $f$ is indeed smooth through the coordinate singularities $s=0,\pi$ and especially $r=0$. 

With $f$ of this form, $\pr_0 f|_\Sigma = c\psi(s)$ may be alternatively written as a polynomial in $\cos^2 s = x_n^2|_\Sigma$. Therefore $x_i\psi$ has odd degree in $x_i$, so by symmetry $\int_\Sigma x_i \psi =0$ and hence $\pr_0 f|_\Sigma$ is orthogonal to the first eigenfunctions $\phi_{1,i}$, for each $i=0,\cdots,n$. 

\subsubsection{Boundary Laplacian} On the boundary $\Sigma$, our choice of $f,\psi$ restricts to \begin{equation}f|_\Sigma =a-b(1-x_n^2)^k |_\Sigma=a-b\sin^{2k}s.\end{equation} With this form the boundary integrals in (\ref{eq:dLap}) may be computed using Lemma \ref{lem:sphint}:

\begin{lemma}
\label{lem:sphint1}
Consider the variation $g=g(t)$, with $g(0)=\ol{g}$ and $\ddt g = f\ol{g}$, with $f$ of the form (\ref{eq:fform}). Then with respect to the basis $\{\phi_{1,i}\}$ of $V_n$, the variation of the Laplacian acts on $V_n$ as $\wh{\Lap}'_{V_n} = \diag(-\mu_1,\cdots,-\mu_n)$, where the $\mu_i$ are all equal for $i<n$, and \begin{equation} \sum_{i=0}^n \mu_i = -n(n+1)a +n\frac{B(k+\frac{n}{2},\frac{1}{2})}{B(\frac{n}{2},\frac{3}{2})} b,\end{equation} \begin{equation}\label{eq:dL1n} \mu_n =-na -\frac{b}{2k+n+1}\frac{B(k+\frac{n}{2},\frac{1}{2})}{B(\frac{n}{2},\frac{3}{2})} (k(n-2)-n).\end{equation}
\end{lemma}
\begin{proof}
Recall that $\phi_{1,i} = \frac{x_i}{\sqrt{C_n}}$, where $C_n =\int_{S^n}x_n^2\, dV_{\ol{g}} = B(\frac{n}{2},\frac{3}{2}) \gamma_{n-1}$. Since $f|_\Sigma$ can be written as the restriction of a polynomial in $x_n$ only, for $0\leq i<j\leq n$ we see that $\phi_{1,i}\phi_{1,j} f|_\Sigma$ is an odd function with respect to $x_i$. Therefore $\int_\Sigma \phi_{1,i}\phi_{1,j}f=0$, so by Corollary \ref{cor:dL1} we thus have $\langle \phi_{1,i},\wh{\Lap}'\phi_{1,j} \rangle = \langle \phi_{1,j},\wh{\Lap}'\phi_{1,i}\rangle =0$. Hence $\wh{\Lap}'_{V_n}$ is indeed diagonal.

Now since $f=a-b(1-x_n^2)^k$ on $\Sigma$, using Corollary \ref{cor:dL1} again, together with Lemma \ref{lem:sphint} and the beta function identity $B(x,y+1) = B(x,y) \frac{y}{x+y},$ gives the formula for $\mu_n = - \langle \phi_{1,n},\wh{\Lap}'\phi_{1,n}\rangle$.

By symmetry of $f|_\Sigma$, it is clear that the $\mu_i = - \langle \phi_{1,i},\wh{\Lap}'\phi_{1,i}\rangle$ are all equal for $i<n$. Now on $\Sigma$ we have $\sum_{i=0}^n \phi_{1,i}^2 = \frac{1}{C_n}\sum_{i=0}^n x_i^2 = \frac{1}{C_n}.$ Then by Lemma \ref{lem:conformal}, we have $ \sum_{i=0}^n \mu_i= -\frac{n}{C_n}\langle f,1\rangle,$ since the $\wh{\Lap}f$ term integrates to zero. Again using Lemma \ref{lem:sphint} we find that indeed $\sum_{i=0}^n \mu_i = -n(n+1)a +n\frac{B(k+\frac{n}{2},\frac{1}{2})}{B(\frac{n}{2},\frac{3}{2})} b.$ 
\end{proof}

Our main aim is to choose $f$ so that $\mu_n<0$. Observe that the second term in (\ref{eq:dL1n}) may be made negative by choosing large enough $k$, so long as $n>2$. When $n=2$, this term is instead always positive, and in fact increases with $k$. For this reason, we treat these two cases separately in the sections to follow. 

\subsubsection{Ricci curvature} Now to analyse the variation of Ricci curvature, we use our knowledge of the Christoffel symbols (\ref{eq:christoffel}). Applying them iteratively, it is a straightforward computation to find the Hessian of $f$: 

\begin{lemma}
\label{lem:hessf}
Let $f\in C^{\infty}(M)$ be of the form (\ref{eq:fform}). Then the Hessian $\nabla^2 f$ can be given in the block form 
\begin{equation} \nabla^2 f = \left(\begin{matrix} \ddot{F} \psi & \dot{\psi} (\dot{F}-F\cot r)&0 \\ \dot{\psi}(\dot{F}-F\cot r)& F\ddot{\psi} + \dot{F}\psi \sin r \cos r & 0 \\0 & 0&\left(\begin{matrix}F  \dot{\psi}\sin s \cos s \\+\dot{F}\psi\sin r \cos r \end{matrix}\right)\ol{g}_{S^{n-1}}    \end{matrix}\right).\end{equation}
In particular, we have \begin{equation}\Lap f= \ddot{F}\psi + \frac{F}{\sin^2 r} \ddot{\psi} + n\dot{F}\psi \cot r + (n-1) \frac{F}{\sin^2 r}\dot{\psi}\cot s .\end{equation}
\end{lemma}

Recall from Lemma \ref{lem:conformal} that with $g'=f\ol{g}$, we have $-2(\Ric'_g - nf\ol{g}) = (\Lap f+2nf)\ol{g} + (n-1)\nabla^2 f$. Thus the condition $( g^{-1} \Ric_g )'= g^{-1} (\Ric'-nf\ol{g}) \geq 0$ amounts to showing that the endomorphism \begin{equation}(\Lap f + 2n f)\id + (n-1) \ol{g}^{-1}\nabla^2 f\end{equation} is negative semidefinite. By Lemma \ref{lem:hessf}, to verify this negativity at a point $p\in M$ it is sufficient to prove the following three inequalities:

\begin{equation}
\label{eq:hess1}
E_1:=\Lap f + 2nf + (n-1) \left(\frac{F}{\sin^2 r}\dot{\psi}\cot s  +\dot{F}\psi \cot r \right) <0,
\end{equation}
\begin{equation}
\label{eq:hess2}
E_2:=\Lap f + 2nf + (n-1) \ddot{F}\psi \leq 0
\end{equation}
\begin{eqnarray}
\label{eq:hess3}D&:=(\Lap f+2nf)^2 + (n-1)(\Lap f +2nf)(\ddot{F}\psi + \frac{F}{\sin^2 r}\ddot{\psi} + \dot{F}\psi\cot r )& \\ \nn&+ (n-1)^2\ddot{F}\psi (\frac{F}{\sin^2 r}\ddot{\psi} + \dot{F}\psi\cot r )- \frac{(n-1)^2}{\sin^2 r} \dot{\psi}^2 (\dot{F}-F\cot r)^2&\geq 0
\end{eqnarray}
These correspond to the lower right block, the upper left entry, and the determinant of the upper left block respectively. The $\Lap f+2nf$ term will be somewhat easier to handle, so we have avoided expanding it explicitly here.

\subsubsection{The case $n\geq3$}
\label{sec:n3}

For $n\geq 3$, we consider $b>0$ and then take $k>\frac{n}{n-2}$ so that the second term in (\ref{eq:dL1n}) is negative. Then we will not need the scaling constant $a$, so we set $a=0$. Lemma \ref{lem:sphint1} then ensures that $\mu_n<0$ and moreover that $\sum_{i=0}^n \mu_i>0$, hence $ \mu_i >0$ for each $i<n$.

Our remaining strategy in this case is to choose $F(r) = r^{2m}$ for sufficiently large $m$, in order to ensure that the variation of Ricci curvature is nonnegative. Note that with this choice of $F$ we indeed have $b= (\pi/2)^{2m}>0$ and $c=2m(\pi/2)^{2m-1}>0$ . 

To prove that the variation of Ricci curvature is nonnegative, by continuity it suffices to verify the conditions (\ref{eq:hess1}), (\ref{eq:hess2}) and (\ref{eq:hess3}) away from the coordinate singularities at $r=0$ and $s=0,\pi$. Thus for the remainder of this section, we will assume $0<r\leq \pi/2$, $0<s<\pi$. 

For convenience, we set $L= \frac{\Lap f + 2nf}{r^{2m-2}\sin^{2k} s}.$ A straightforward calculation shows that, with our choice of $f$, \begin{equation} L = -(2m)(2m-1) - 2mnr\cot r + \frac{2kr^2}{\sin^2 r} - \frac{2k(2k+n-2)r^2\cot^2 s}{\sin^2 r} - 2nr^2.\end{equation} We will choose $m$ large enough so that $(2m-1) > \frac{\pi^2 k}{2}$. Since $\sin r \geq \frac{2r}{\pi}$ for $r\in [0,\pi/2]$, we then have \begin{equation}\label{eq:mkest}(2m-1)\sin^2 r \geq 2kr^2,\end{equation} which easily gives \begin{equation}\label{eq:Lest}L<-(2m-1)^2 - \frac{2k(2k+n-2)r^2\cot^2 s}{\sin^2 r}<0\end{equation} and hence $\Lap f + 2nf <0.$

Further calculations then give \begin{equation} E_1 = \Lap f + 2nf - 2(n-1)r^{2m}\sin^{2k}s\left(\frac{k\cot^2 s}{\sin^2 r}  +  \frac{m\cot r}{r}\right) <0 ,\end{equation} \begin{equation}E_2 = \Lap f + 2nf -(n-1)2m(2m-1)r^{2m-2} \sin^{2k}s <0.\end{equation}

Again a straightforward calculation gives \begin{eqnarray}\nn\frac{D\sin^2 r}{r^{4m-4}\sin^{4k} s} &=& L^2 \sin^2 r + 4(n-1)^2 k^2 \cot^2 s (4mr^3 \cot r-r^4 \cot^2 r )\\&&\nn+2(n-1)L(-m(2m-1)\sin^2 r + kr^2)\\&&\nn - 2(n-1)L(  k(2k-1)r^2\cot^2 s + mr\sin r\cos r)\nn\\&&\nn -4(n-1)^2 m(2m-1)(kr^2- mr\sin r \cos r) \\&&\label{eq:Dn3} -4(n-1)^2 km(2k+2m-1) r^2 \cot^2 s .\end{eqnarray}

Since $\tan r \geq r$ for $0\leq r <\pi/2$, we have $r\cot r \leq 1$. So in particular $r^4 \cot^2 r \leq 4mr^3\cot r,$ and hence the first line is positive. 

The remaining negative terms we must handle are the last line and the $kr^2$ in the fourth line. We can estimate them as follows:

By (\ref{eq:mkest}) and (\ref{eq:Lest}), we have \begin{eqnarray}\nn 2L(-m(2m-1)\sin^2 r + kr^2) &>&(2m-1)^4 \sin^2 r \\&& +2 (2m-1)^2 k(2k+n-2)r^2\cot^2 s. \end{eqnarray}

For $2m-1\geq 2(n-1)$, we have \begin{eqnarray}  (2m-1)^4 \sin^2 r &>&\nn 2(2m-1)^3 kr^2 \geq 2m(2m-1)^2 kr^2  \\&\geq& 4(n-1) m(2m-1)kr^2. \end{eqnarray}

Finally, we will have \begin{equation}2 (2m-1)^2 k(2k+n-2)r^2\cot^2 s \geq 4(n-1) km(2k+2m-1) r^2 \cot^2 s\end{equation} so long as \begin{equation}\label{eq:mkcond}\frac{(2m-1)^2}{2m(2k+2m-1)}   \geq \frac{n-1}{n+2k-2}.\end{equation} For fixed $k > \frac{n}{n-2}>1$, the right hand side of (\ref{eq:mkcond}) is strictly less than 1, whilst the left hand side tends to 1 as $m\rightarrow \infty$, so this condition is satisfied for large $m$.

With these estimates, equation (\ref{eq:Dn3}) implies that $D>0$, and thus we have proven

\begin{proposition}
Let $k>\frac{n}{n-2}$ be a positive integer. Then there is a positive integer $m$ such that
\begin{equation}2m-1\geq \max(2(n-1),8k/\pi^2)\qquad \text{and}\qquad
\frac{(2m-1)^2}{2m(2k+2m-1)}> \frac{n-1}{n+2k-2}.\end{equation}
With this choice of $k,m$, the function $f = -r^{2m}\sin^{2k}s$ satisfies all the properties of Proposition \ref{prop:explicit}.
\end{proposition}

\begin{remark}
The above choice of $m$ might be far from optimal. Some quick plots suggest that, for example, when $5\leq n\leq 11$ it is sufficient to take $k=2$, $m=4$. 
\end{remark}

\subsubsection{The case $n=2$}
\label{sec:n2}

When $n=2$, we take $k=1$, so that Lemma \ref{lem:sphint1} gives \begin{equation}\mu_n= -2a+\frac{4}{5}b, \qquad \text{and}\qquad \sum_{i=0}^n  \mu_i = -6a + 4b.\end{equation} We thus choose the scaling constant $a= (\frac{2}{5}+\epsilon_0)b >0 $ for a small $0<\epsilon_0 < 4/15$, so that $\mu_{n}<0$ and $\sum_{i=0}^n  \mu_i >0$, hence $ \mu_i >0$ for $i<n$. For our analysis we make the specific choice $\epsilon_0 = 10^{-6}.$

Our strategy in this case is to find a particular $F$ for which the variation of Ricci curvature is still nonnegative, despite the positive scaling $a$. Specifically, we choose the function \begin{equation} \label{eq:bigF}F(r)= \frac{1}{C}\left(r^2 - \frac{1}{21} r^4 + \frac{4}{315} r^6 + \frac{1}{945} r^8 + \frac{74}{429925} r^{10}\right),\end{equation} where the constant $C$ is chosen so that $b=F(\pi/2)=1$. Note also that indeed $c= \dot{F}(\pi/2) = 1.416\cdots >0.$

\begin{remark}
The polynomial above is, up to a normalising constant, the tenth degree Taylor polynomial of the solution of $\ddot{y} + 2\dot{y}\cot r +4y -\frac{6y}{\sin^2 r}=0,$ with boundary conditions $y(0)=0, y(\pi/2)=1$. This ODE arises when finding harmonic extensions of functions $z$ on $\pr S^3_+$, when $z$ is expanded in terms of spherical harmonics on $S^2$. In particular, it governs the $\lambda = 6$ (second degree) eigenspace.

The solution $y$ may be written in terms of associated Legendre functions $P_\nu^\mu$ as \begin{equation}y(r) = \frac{P_{-\frac{1}{2}+\sqrt{5}}^{-\frac{5}{2}}(\cos r)}{P^{-\frac{5}{2}}_{-\frac{1}{2}+\sqrt{5}}(0)\sqrt{\sin{r}}}.\end{equation} A plot of the relevant quantities (\ref{eq:hess1}), (\ref{eq:hess2}) and (\ref{eq:hess3}) suggested that $y$ was also a (possibly more natural) candidate for the function $F$. However, at present the author is not aware of a proof based on the above interpretation of $y$, and the analysis was simpler with the polynomial form of $F$. 
\end{remark}

We will verify the conditions (\ref{eq:hess1}), (\ref{eq:hess2}) and (\ref{eq:hess3}) at each point of $M$. The quantities $E_1$ and $E_2$ are manageable: A straightforward computation gives that \[ E_1 = 4a - \cos^2 s \left(\frac{6F}{\sin^2 r}\right) - \sin^2 s \left( \ddot{F}+3\dot{F}\cot r + 4F - \frac{2F}{\sin^2 r} \right),\] \[E_2 = 4a -\cos^2 s \left(\frac{4F}{\sin^2 r}\right) -2\sin^2 s\left(\ddot{F}+\dot{F}\cot r + 2F - \frac{F}{\sin^2 r}\right).\] 

We thus require some relatively elementary, but tedious properties of $F$, which we defer to the appendix (see points (vi-viii) of Lemma \ref{lem:trivialest}): 

\begin{lemma}
On $[0,\frac{\pi}{2}]$, the function $F(r)$ satisfies the following properties:
\begin{itemize}
\item $\frac{F}{\sin^2 r} \geq 0.41 $,
\item $\ddot{F}+3\dot{F}\cot r + 4F - \frac{2F}{\sin^2 r} \geq 1.9$,
\item $\ddot{F}+\dot{F}\cot r + 2F - \frac{F}{\sin^2 r}\geq 1.1$.
\end{itemize}
\end{lemma}

With our choice of $a=0.400001$, it follows that $E_1\leq 4a -\max(6(0.41),1.9) <0$ and $E_2\leq 4a - \max(4(0.41),2(1.1))<0$ for all $0\leq r\leq \pi/2$, $0\leq s\leq \pi$. 

The analysis of $D$ is significantly more complicated. Our proof relies on obtaining crude bounds for the derivative, $|\pr_s D| \leq 80$, $|\pr_r D|\leq 202$, and numerically calculating values of $D$ at points $(r,s)$ on a sufficiently fine grid. In particular, the function $D$ was sampled on a square grid with spacing $\delta = 10^{-4}$. Across all sampled points, the minimum value of $D$ was found to be $0.01536\cdots.$ The mean value theorem then implies that \[D> 0.015 - 202 \frac{\delta}{\sqrt{2}} = 0.0007\cdots >0\] for all $0\leq r\leq \pi/2$, $0\leq s\leq \pi$. Further details are again left to the appendix.

\begin{remark} For any choice of $F$, the function $D$ is quadratic in $\sigma=\sin^2 s$, say $D= a_2\sigma^2 + a_1\sigma + a_0$. It is possible that the estimate $D\geq 0$ could be manageable without the numerically driven, but still mathematically rigorous, computation presented above and in the appendix. Indeed, bounds similar to those in the above lemma establish that $D\geq 0$ for the special values $s=0,\pi/2,\pi$. Setting $\pr_s D=0$ gives one further possibility for critical points, namely when $\sigma=\sin^2 s = -a_1/2a_2$. A plot of $-a_1/2a_2$ as a function of $r\in [0,\pi/2]$ indicates that it is significantly larger than 1, which would complete the analysis since of course $\sin^2 s$ must be at most 1. 
\end{remark}

\subsection{Correcting the mean curvature}
\label{sec:mainpf}

In this section, we complete the proof of Theorem \ref{thm:main} by first perturbing in the direction of an ambient diffeomorphism that corrects the mean curvature to first order, and then introducing a lower order conformal correction that fixes the mean curvature exactly at zero. We will verify that the resulting metric provides a counterexample to Conjecture \ref{question} for manifolds with boundary, and satisfies the desired properties of Theorem \ref{thm:main} (possibly after scaling). 

For this section we fix a smooth cutoff function $\chi:[0,\pi]\rightarrow [0,100]$ such that:
\begin{itemize}
\item $\chi(r)=0$ for $r\leq \pi/3$ and $r\geq 2\pi/3$.
\item $\dot{\chi}(\pi/2)=1$. 
\end{itemize}

\begin{proof}[Proof of Theorem \ref{thm:main}]
Let $M=S^{n+1}_+$ and consider the conformal factor $f\in C^\infty(M)$ as in Section \ref{sec:explicit}. The property that $\pr_0 f|_\Sigma$ is orthogonal to $V_n=\ker (\wh{\Lap}+n)$ is crucial: It means that there exists a smooth function $v\in C^\infty(\Sigma)$ such that \begin{equation}\label{eq:lapv}(\wh{\Lap}+n) v= \frac{n}{2}\pr_0 f|_\Sigma.\end{equation} 

Recalling that $\pr_0 f |_\Sigma = -c\sin^{2k}s$ and using that $\wh{\Lap}$ acts on the class of functions depending only on $s$ by $\pr_s^2 + (n-1)\cot s\, \pr_s$, it is easily verified that an explicit solution to equation (\ref{eq:lapv}) is given by \begin{equation}\label{eq:explicitv} v= \frac{nc}{2(2k-1)(n+2k)}  \sum_{j=0}^k a_j \sin^{2j} s ,\end{equation} where the coefficients satisfy $  a_j = \frac{2(j+1)}{2j-1}a_{j+1}$, $a_k=1$.

Then we may fix a (smooth) 1-form $\omega\in\Omega^1(M)$ for which $\omega_0|_\Sigma =v.$ Explicitly, we may take for instance $\omega(r,\theta) = \dot{\chi}(r)v(\theta)dr$, where $r$ is as in the warped product (\ref{eq:warp1}), and $\theta$ are the coordinates on $S^n$. For small $t$ consider the (analytic) variation \begin{equation}g_1(t) = \ol{g} + tf\ol{g} + t\mathcal{L}\omega.\end{equation}

For $u\in C^\infty(S^n)$, we also define a smooth extension map $\mathcal{E}:C^\infty(S^n)\rightarrow C^\infty(S^{n+1})$ by $\mathcal{E}(u)(r,\theta) = \chi(r)u(\theta).$ This extension is constructed so that $\pr_0 \mathcal{E}(u)|_\Sigma = u$. 

Now consider the family of functions on $\Sigma$ defined by \begin{equation}u(t) = -\frac{2}{n}H(\Sigma, g_1(t)).\end{equation} Since the mean curvature functional depends analytically on the metric and its derivatives, $u$ is analytic in $t$. Moreover, by Lemmas \ref{lem:diffeo} and \ref{lem:conformal}, we have that \begin{equation}u' = \frac{2}{n} (\wh{\Lap}+n)v  - \pr_0f|_\Sigma= 0.\end{equation}

We may thus consider the real analytic family of smooth metrics \begin{equation}g(t) = e^{\mathcal{E}(u(t))} g_1(t)=e^{\mathcal{E}(u(t))}(\ol{g} + tf\ol{g} + t\mathcal{L}\omega)\end{equation} on $M$. Note that indeed $g(0)=\ol{g}$. 

By the formula (\ref{eq:Hconf}) for the mean curvature under conformal change, we have \begin{equation}H(\Sigma, g(t)) = e^{-\mathcal{E}(u(t))/2} \left( H(\Sigma, g_1(t)) + \frac{n}{2} u(t)\right) =0.\end{equation}

Now since $u'=0$, we have $\mathcal{E}(u)'=0$, and therefore the first variation of the metric $g(t)$ does not depend on $u$. In particular, on $M$ we have  \begin{equation}g' =g'_1= f\ol{g} + \mathcal{L}\omega.\end{equation} Again by Lemma \ref{lem:diffeo}, the diffeomorphism part $\mathcal{L}\omega$ does not affect the Ricci curvature nor the spectrum of $\wh{\Lap}$ on $\Sigma$ to first order. Thus the conclusions of Proposition \ref{prop:explicit} still hold, namely that the Ricci curvature on $M$ is nondecreasing, \begin{equation}(g(t)^{-1} \Ric_{g(t)})' \geq 0,\end{equation} and that the variation of the Laplacian on $\Sigma$ acts on $V_n$ by \begin{equation} \wh{\Lap}'_{V_n} = \diag(-\mu_1,\cdots,-\mu_n),\end{equation} with respect to the basis $\{\phi_{1,i}\}$, where \begin{equation}\mu_i>0\qquad \text{for $i<n$}, \qquad\text{and}\qquad \mu_n<0.\end{equation}

At this point we would like to conclude that the first eigenvalue $\lambda_1(\Sigma)$ varies by the $\mu_i$, but formula (\ref{eq:dL1}) will only apply if we already know that the $\phi_{1,i}$ are initial points of some smoothly varying families of eigenfunctions of $\wh{\Lap}_{\wh{g}(t)}$. The key is that $\wh{\Lap}'_{V_n}$ is already diagonal in the basis $\{\phi_{1,i}\}$: 

By Lemma \ref{lem:analeig}, there are families of $L^2(\Sigma,\wh{g}(t))$-orthonormal eigenfunctions of $\wh{\Lap}_{\wh{g}(t)}$, suggestively denoted $\{\varphi_{1,i}(t)\}_{i=0}^n$ with corresponding eigenvalues $\lambda_{1,i}(t)$, varying analytically in $t$, where $\lambda_{1,i}(0)=n$. By the discussion of Section \ref{sec:varlaplace}, the variation $\wh{\Lap}'_{V_n}$ must be diagonal in the basis $\{\varphi_{1,i}(0)\}$. But there is at most one orthonormal basis that diagonalises a matrix, up to sign and permutation, so we can indeed arrange that $\varphi_{1,i}(0)=\phi_{1,i}$, and hence \begin{equation}\lambda'_{1,i} = \mu_i\end{equation} for each $i$. (This claim would also follow from a symmetry argument, noting that the explicit form of the variation $\wh{g}'$ is invariant under rotations fixing $\phi_{1,n}$.) 

Finally, we claim that the variation $\left( \int_\Sigma A(\wh{\nabla} \varphi_{1,n}, \wh{\nabla}\varphi_{1,n}) \, dV_{\wh{g}(t)}\right)'> 0.$ Recall that we chose coordinates so that $e_1 = \pr_s$, and also that the eigenfunction $\phi_{1,n} = \frac{x_n|_\Sigma}{\sqrt{C_n}}  = \frac{\cos s}{\sqrt{C_n}}$ only depends on $s$. Then since $A(\Sigma,\ol{g})=0$, we have \begin{equation} A(\wh{\nabla} \varphi_{1,n}, \wh{\nabla}\varphi_{1,n})' =  \left(A^{\alpha\beta} \wh{\nabla}_\alpha \varphi_{1,n} \wh{\nabla}_\beta\varphi_{1,n} \right)'  = (\pr_1 \phi_{1,n})^2 A'_{11}.\end{equation}

Now by Lemmas \ref{lem:diffeo} and \ref{lem:conformal}, and since $v$ satisfies equation (\ref{eq:lapv}), we have \begin{equation}A_{11}' = -\pr_s^2 v -v + \frac{1}{2}\pr_0 f|_\Sigma = \frac{n-1}{n}(\cot s\, \pr_s -\pr_s^2 )v.\end{equation}  Using the explicit form (\ref{eq:explicitv}) of $v$, we compute that $\cot s\, \pr_s v-\pr_s^2v = \frac{nkc}{n+2k}\sin^{2k} s.$ 

Since $A=A(\Sigma,\ol{g})=0$, noting that $(\pr_1 \phi_{1,n})^2 = \frac{\sin^2 s}{C_n}$ and using Lemma \ref{lem:sphint} we indeed have \begin{eqnarray}\nn \left( \int_\Sigma A(\wh{\nabla} \varphi_{1,n}, \wh{\nabla}\varphi_{1,n}) \, dV_{\wh{g}(t)}\right)' &=& \int_\Sigma A(\wh{\nabla} \varphi_{1,n}, \wh{\nabla}\varphi_{1,n})' dV_{\ol{g}} \\&=& \frac{(n-1)kc B(k+1+\frac{n}{2},\frac{1}{2})}{(n+2k)B(\frac{n}{2},\frac{3}{2})} >0.\end{eqnarray}

To finish the construction, choose a small $\epsilon>0$ so that $n\epsilon/2 < - \lambda'_{1,n}$. Since the variation $g(t)$ is analytic, we conclude that for sufficiently small $t>0$ the smooth metric $g(t)$ on $M$ satisfies the following properties:

\begin{itemize}
\item The boundary $\Sigma$ remains minimal: \begin{equation}H(\Sigma, g(t))=0.\end{equation}
\item The Ricci curvature is bounded below by \begin{equation}\Ric_{g(t)} \geq (1-\epsilon t/2)ng(t).\end{equation}
\item The eigenvalues of the Laplacian $\wh{\Lap}$ on $\Sigma$ satisfy \begin{equation}0< \lambda_{1,n}(\Sigma, \wh{g}(t)) < n < \lambda_{1,i}(\Sigma, \wh{g}(t)) <  \cdots,\end{equation} where $i=0,\cdots,n-1$. In particular, the first nonzero eigenvalue is \begin{equation}\lambda_1(\Sigma,\wh{g}(t)) = \lambda_{1,n}(\Sigma,\wh{g}(t)) <(1-\epsilon t/2)n.\end{equation} 
\item The second fundamental form $A=A(\Sigma,g(t))$ satisfies \begin{equation}\int_\Sigma A(\wh{\nabla}\phi_1 , \wh{\nabla} \phi_1) \, dV_{\wh{g}(t)}> 0,\end{equation} where $\phi_1 = \varphi_{1,n}(t)$ is the eigenfunction corresponding to $\lambda_{1,n}(\Sigma,\wh{g}(t))$. 
\end{itemize}

Then the scaled metric \begin{equation}g=(1-\epsilon t/2)g(t)\end{equation} satisfies all the desired properties, and completes the proof, of Theorem \ref{thm:main}.

\end{proof}

\appendix

\renewcommand\thesection{\Alph{section}}
\makeatletter
\renewcommand*{\@seccntformat}[1]{%
{\upshape{\csname addname@#1\endcsname \csname the#1\endcsname.}}\hskip .5em}
\newcommand*{\addname@section}{\appendixname\ }
\makeatletter

\section{Analysis of the conformal factor when $n=2$}

In this appendix we include the details of the analysis of the conformal factor $f = a- F(r) \sin^2 s$ of Section \ref{sec:n2}, in which $n=2$. In particular we give various estimates related to the function (\ref{eq:bigF}), \begin{equation}F(r)= \frac{1}{C}\left(r^2 - \frac{1}{21} r^4 + \frac{4}{315} r^6 + \frac{1}{945} r^8 + \frac{74}{429925} r^{10}\right),\end{equation} and the details of the numerical analysis of the corresponding quantity $D$ in condition (\ref{eq:hess3}). Recall that $C$ is chosen so that $F(\pi/2)=1$; numerically $C=2.423\cdots$. 

\subsection{Basic estimates for $F$}

We only need some crude estimates for $F$, so our strategy is to estimate all the expressions that arise by polynomials. To do this, we use certain estimates for trigonometric functions coming from their standard power series expansions:

\begin{lemma}
\label{lem:trigest}
For $r\in [0,\pi/2]$ we have: \begin{equation}\frac{2r}{\pi} \leq \sin r \leq r, \qquad 1-\frac{r^2}{2}\leq r\cot r \leq 1-\frac{r^2}{3},\qquad 1+\frac{r^2}{3}\leq \frac{r^2}{\sin^2 r} \leq 1+r^2.\end{equation} 
\end{lemma}

First we obtain some bounds on $F$ and its derivatives:

\begin{lemma}
\label{lem:trivialest}
Let $F$ be as above. Then for $r\in [0,\pi/2]$ we have:
\begin{enumerate}[(i)]
\item $0\leq F \leq 1$,
\item $0\leq \dot{F}\leq 1.5$,
\item $0\leq \ddot{F} \leq 2.4$,
\item $-0.75\leq \dddot{F}\leq 5.1$,
\item $0\leq \frac{\dot{F}-2F\cot r}{\sin^2 r}\leq 1.8 $,
\item $\frac{1}{C} \leq \frac{F}{\sin^2 r} \leq 1$,
\item $\ddot{F} + 3\dot{F}\cot r + 4F - \frac{2F}{\sin^2 r} \geq 1.9$,
\item $\ddot{F} + \dot{F}\cot r + 2F-\frac{F}{\sin^2 r} \geq 1.1$,
\item $-5.1 \leq \ddot{F}+ 2\dot{F}\cot r + 4F -\frac{6F}{\sin^2 r} \leq 1.7$,
\item $-2.3 \leq 2\ddot{F}+3\dot{F}\cot r +4F - \frac{4F}{\sin^2 r} \leq 6.1$.
\item $-11 \leq 2\ddot{F} + 3\dot{F}\cot r + 4F - \frac{10F}{\sin^2 r} \leq 3.4$.
\item $0.7\leq \ddot{F} + 2\dot{F}\cot r + 4F - \frac{2F}{\sin^2 r}\leq 4.7$,
\item $-5.2 \leq \dot{F}\cot r - \frac{3F}{\sin^2 r} \leq 0$,
\item $0\leq \frac{\dot{F}-F\cot r}{\sin r} \leq 1.58 $,
\item $0\leq \frac{1}{\sin r}(\ddot{F} - 2\dot{F}\cot r - F + \frac{2F}{\sin^2 r})\leq 4.1$,
\item $-3.6 \leq \ddot{F}\cot r -\frac{\dot{F}}{\sin^2 r}\leq 0$,
\item $0\leq \dot{F}\cot r\leq 0.9$,
\item $-3.3\leq \dddot{F}+2\ddot{F}\cot r +4\dot{F} - \frac{4\dot{F}}{\sin^2 r} + \frac{4F\cot r}{\sin^2 r}\leq 7.1$. 
\end{enumerate}
\end{lemma}
\begin{proof}
Points (iii-iv) follow by directly differentiating, then retaining only the negative nonconstant terms for a lower bound, and the positive nonconstant terms for an upper bound, and evaluating at $r=\pi/2$. For instance, \begin{equation} C\dddot{F}(r) = -\frac{8}{7}r + \frac{32}{21} r^3 + \frac{16}{45} r^5 + \frac{1184}{9555} r^7,\end{equation} but $-\frac{8\pi}{14C} =-0.74\cdots \geq -0.75$ and $\frac{1}{C}\left(\frac{32}{21} \left(\frac{\pi}{2}\right)^3 + \frac{16}{45} \left(\frac{\pi}{2}\right)^5 + \frac{1184}{9555} \left(\frac{\pi}{2}\right)^7\right) = 5.04\cdots \leq 5.1$.

Points (i-ii) follow from point (iii), which implies that $\dot{F}$ is increasing. Points (v, vii-xvi, xviii) are estimated in the same way as (iii-iv), after replacing the trigonometric functions using Lemma \ref{lem:trigest}. For point (vi), note that $\frac{d}{dr} \frac{F}{\sin^2 r} = \frac{\dot{F}-2F\cot r}{\sin^2 r},$ so point (i) implies that $\frac{F}{\sin^2 r}$ is increasing, which gives the desired bounds. Similarly, for point (xvii), we have $\frac{d}{dr} (\dot{F}\cot r) = \ddot{F}\cot r - \frac{\dot{F}}{\sin^2 r}$, so point (xvi) implies that $\dot{F}\cot r$ is decreasing. Checking the endpoints again gives the desired bounds.
\end{proof}

To shorten the formulae, we define $Q_{j}$ to be the quantity estimated in the respective points of Lemma \ref{lem:trivialest}. Finally, we record separately some estimates for $\Lap f+ 2nf = \Lap f + 4f$. 

\begin{lemma}
\label{lem:lapest}
Let $f,F$ be as above, with $a=0.400001$. Then for any $r\in [0,\pi/2]$, $s\in [0,\pi]$ we have \begin{equation}|\Lap f + 4f| \leq 3.1\qquad \text{and}\qquad\left|\pr_r(\Lap f + 4f) \right| \leq 7.2.\end{equation}
\end{lemma}
\begin{proof}
We compute $\Lap f + 4f = 4a -Q_{12}\sin^2 s -2Q_6\cos^2 s .$ Estimating the convex combination by points (vi) and (xii) of Lemma \ref{lem:trivialest} gives the bounds as claimed. We may also compute $\pr_r(\Lap f+4f) =\nn -Q_{18}\sin^2 s- 4Q_5\cos^2 s\left( \frac{4\dot{F}-8F\cot r}{\sin^2 r}\right).$ Arguing as above, points (v) and (xviii) of Lemma \ref{lem:trivialest} give the desired bounds.
\end{proof}

\subsection{Derivative estimates for $D$}

We use the estimates of the previous section to bound the derivatives of $D$.

\begin{proposition}
Let $F$ be as above, with $a=0.400001$, and let $D$ be as in (\ref{eq:hess3}). Then we have $\left|\frac{\pr D}{\pr s}\right|\leq 80$ and $\left|\frac{\pr D}{\pr r}\right|\leq 202.$
\end{proposition}
\begin{proof}
Plugging directly into (\ref{eq:hess3}), we have \begin{eqnarray}\nn D &=& \left(4a -Q_{12}\sin^2 s -4Q_6\cos^2 s \right)\left(4a -Q_{10}\sin^2 s -6Q_6\cos^2 s \right)\\&& - Q_3\sin^2 s \left( 2Q_6(2\sin^2 s-1) - Q_{17} \sin^2 s\right) -4Q_{14}^2\sin^2 s \cos^2 s.\end{eqnarray}

Taking the derivative in $s$ gives \begin{eqnarray}\nn -\frac{\pr_s D}{\sin 2s} =&& Q_9 \left(4a -Q_{10}\sin^2 s -6Q_6\cos^2 s \right) +Q_3\left(-2Q_{13}\sin^2 s -2Q_6\cos^2 s\right) \\&&  +Q_{11}  \left(4a -Q_{12}\sin^2 s -4Q_6\cos^2 s\right)+ 4Q_{14}^2\cos 2s . \end{eqnarray}

Using points (iii, vi, ix-xiv) of Lemma \ref{lem:trivialest}, and bounding each term as in Lemma \ref{lem:lapest} gives the claimed result for $|\pr_s D|$. Taking the derivative in $r$ yields \begin{eqnarray}\nn \pr_r D &=&  -(\Lap f + 4f) \left( 2Q_5\cos 2s + \left(Q_4+Q_{16}\right)\sin^2 s\right)-2Q_{14} Q_{15} \sin^2 2s\\&&\nn+\pr_r(\Lap f+4f)\left(2(\Lap f +4f) -(Q_3+Q_{17})\sin^2 s   - 2Q_6 \cos 2s\right)\\&& +Q_4 \left(2Q_6\cos 2s + Q_{17}\sin^2 s\right)\sin^2 s \nn\\&&+Q_3\left( 2Q_5 \cos 2s + Q_{16} \sin^2 s\right)\sin^2 s.\end{eqnarray}

Lemma \ref{lem:lapest} and points (iii-vi, xv-xvii) of Lemma \ref{lem:trivialest} then give the bounds for $|\pr_r D|$.

\end{proof}

\subsection{Proof that $D\geq 0$}

We now give a computationally aided proof that the condition (\ref{eq:hess3}) is satisfied for $f = a-F(r)\sin^2 s$, where $a=0.400001$ and \begin{equation}F(r)= \frac{1}{C}\left(r^2 - \frac{1}{21} r^4 + \frac{4}{315} r^6 + \frac{1}{945} r^8 + \frac{74}{429925} r^{10}\right)\end{equation} as above. Note that $D$ is smooth as a function of two variables $r,s$, and is moreover symmetric across $s=\pi/2$, so it is enough to consider $D$ as a smooth function $[0,\pi/2]^2 \rightarrow \mathbb{R}$. By the above proposition we have $|\grad D|\leq 202$. 

We used the software Maple (version 18.01) to sample values of $D$ at a square grid covering the region $R=[0,\pi/2]^2$, with grid spacing $\delta = 10^{-4}$. The working precision was set to 15 significant figures. The minimum value of $D$ over the sampled grid points was found (at the point $r=0.76488\cdots $, $s=1.5708\cdots \simeq \pi/2$) to be $m=0.01536\cdots$. 

At any point $(r,s)\in R$, the Euclidean distance to a sampled grid point is at most $\frac{\delta}{\sqrt{2}}$. Then by the mean value theorem we have \begin{eqnarray}\nn D(r,s) &\geq& m - |\grad D| \frac{\delta}{\sqrt{2}} \geq 0.015 - \frac{202}{\sqrt{2}}(0.0001) \\&\geq& 0.0007 >0.\end{eqnarray}

\subsection{Code used for computation}

Here we include the Maple code used for the computations of the previous sections.

\lstset{basicstyle=\small, language=Matlab}

\begin{lstlisting}
Digits:=15; n:=2: a:=0.400001:
F:=r->(r^2-r^4/21+4*r^6/315+r^8/945+74*r^10/429975):
C:=evalf(F(Pi/2)):
f:=a-sin(s)^2*F(r)/C:
Lf:=diff(f,r,r)+1/sin(r)^2*diff(f,s,s)+cot(r)*diff(f,r)):
   +cot(r)*diff(f,r)+(n-1)*(cot(s)*diff(f,s)/sin(r)^2
eq1:=(Lf+2*n*f)
   +(n-1)*(cot(s)*diff(f,s)/sin(r)^2+cot(r)*diff(f,r)):
eq2:=Lf+2*n*f+(n-1)*diff(f,r,r):
DD:=eq2*((Lf+2*n*f)+(n-1)*(1/sin(r)^2*diff(f,s,s)
   +cot(r)*diff(f,r)))
   -(n-1)^2/sin(r)^2*(diff(f,r,s)-cot(r)*diff(f,s))^2:
\end{lstlisting}
\begin{lstlisting}
minDD:=proc(delta)
   m:=1: mr:=-1: mt:=-1:
   for A from 1 to floor(Pi/2/delta)+1 do
       DDr:=simplify(subs(r=A*delta,DD)):
       for B from 1 to floor(Pi/2/delta)+1 do
           p:=evalf(subs(s=B*delta,DDr)):
           if p<m then m:=p: mr:=A*delta: mt:=B*delta:
           fi; od; od; 
return m; end proc:
minDD(0.0001);
\end{lstlisting}


\bibliographystyle{plain}
\bibliography{sphdefv9a}

\end{document}